\documentclass[12pt]{amsart}
\usepackage{graphicx,amssymb}
\usepackage[usenames]{color}
\usepackage{amsthm,amsfonts,amsmath,amssymb,latexsym,epsfig,mathrsfs,yfonts,marvosym,latexsym,epsfig}
\usepackage{graphicx}
\usepackage[all]{xy}
\usepackage{epsfig}

\AtEndDocument{\vfill\eject\batchmode}
\DeclareMathAlphabet\oldmathcal{OMS}        {cmsy}{b}{n}
\SetMathAlphabet    \oldmathcal{normal}{OMS}{cmsy}{m}{n}
\DeclareMathAlphabet\oldmathbcal{OMS}       {cmsy}{b}{n}
\usepackage[active]{srcltx}
\usepackage{eucal}

\newtheorem{theorem}{Theorem}[section]
\newtheorem{lemma}[theorem]{Lemma}
\newtheorem{proposition}[theorem]{Proposition}
\newtheorem{corollary}[theorem]{Corollary}

\newtheorem{def/prop}[theorem]{Definition/Proposition}

\newtheorem{conjecture}[theorem]{Conjecture}

\theoremstyle{definition}
\newtheorem{definition}[theorem]{Definition}
\newtheorem{remark}[theorem]{Remark}
\newtheorem*{ack}{Acknowledgements}
\newtheorem{example}{Example}[section]

\DeclareSymbolFont{bbold}{U}{bbold}{m}{n}
\DeclareSymbolFontAlphabet{\mathbbold}{bbold}

\def\BOne{\mathchoice{\scalebox{1.16}{$\displaystyle\mathbbold 1$}}{\scalebox{1.16}{$\textstyle\mathbbold 1$}}{\scalebox{1.16}{$\scriptstyle\mathbbold 1$}}{\scalebox{1.16}{$\scriptscriptstyle\mathbbold 1$}}}
\def\fract#1#2{\raise4pt\hbox{$ #1 \atop #2 $}}

\def\bbc{{\mathbb C}}

\def\bbp{{\mathbb P}}
\def\bbq{{\mathbb Q}}
\def\bbr{{\mathbb R}}

\def\bbt{{\mathbb T}}

\def\bbz{{\mathbb Z}}

\def\gra{\alpha}
\def\grb{\beta}

\def\grg{\gamma}

\def\gro{\omega}

\def\grs{\sigma}
\def\grt{\tau}

\def\grD{\Delta}

\def\bfl{{\bf l}}
\def\bfm{{\bf m}}
\def\bfn{{\bf n}}

\def\bfv{{\bf v}}
\def\bfw{{\bf w}}

\def\calb{{\mathcal B}}

\def\cald{{\mathcal D}}
\def\cale{{\mathcal E}}

\def\calh{{\mathcal H}}
\def\cali{{\mathcal I}}
\def\calj{{\mathcal J}}
\def\calk{{\mathcal K}}

\def\calo{{\mathcal O}}

\def\cals{{\oldmathcal S}}

\def\gt{{\mathfrak t}}
\def\gu{{\mathfrak u}}

\def\gA{{\mathfrak A}}

\def\lcm{{\rm lcm}}

\def\tf{\tilde{f}}

\def\<{\langle}
\def\>{\rangle}
\def\ra#1{\to}

\def\fract#1#2{\raise4pt\hbox{$ #1 \atop #2 $}}
\def\decdnar#1{\phantom{\hbox{$\scriptstyle{#1}$}}
\left\downarrow\vbox{\vskip15pt\hbox{$\scriptstyle{#1}$}}\right.}

\def\lcm{{\rm lcm}}

\def\hook{\mathbin{\hbox to 6pt{%
                 \vrule height0.4pt width5pt depth0pt
                 \kern-.4pt
                 \vrule height6pt width0.4pt depth0pt\hss}}}

\begin{document}

\title{Iterated $S^3$ Sasaki Joins and Bott Orbifolds}

\author[Charles Boyer]{Charles P. Boyer}
\address{Charles P. Boyer, Department of Mathematics and Statistics,
University of New Mexico, Albuquerque, NM 87131.}
\email{cboyer@math.unm.edu} 
\author[Christina T{\o}nnesen-Friedman]{Christina W. T{\o}nnesen-Friedman}
\address{Christina W. T{\o}nnesen-Friedman, Department of Mathematics, Union
College, Schenectady, New York 12308, USA }
\email{tonnesec@union.edu}
\thanks{The authors were partially supported by grants from the
Simons Foundation, CPB by (\#519432), and CWT-F by (\#422410)}
\date{\today}

\maketitle
\vspace{-7mm}



\section*{Introduction}\label{intro}
It is well known that to a compact quasi-regular Sasaki manifold (orbifold) one can uniquely associate a projective algebraic variety with an orbifold K\"ahler structure. To understand this relation better it behooves us to think in categorical terms as a functor between the category of compact quasi-regular Sasaki manifolds (orbifolds) and the category of projective algebraic varieties with an orbifold K\"ahler structure. Here we restrict ourselves to the case where the objects in the Sasaki category are what we call {\it iterated $S^3_\bfw$ Sasaki joins} which are obtained by iterating the $S^3_\bfw$ join construction described in \cite{BoTo19a}. The objects in the K\"ahler category are {\it orbifold Bott towers} which are Bott towers \cite{GrKa94,BoCaTo17} with an orbifold structure on the invariant divisors. At first glance one might expect a map from the Sasaki cone to the K\"ahler cone, but there is no such map. What there is, is a map of topoi of groupoids, that is a functor, and it is the purpose of this note to investigate this relationship with the goal of producing constant scalar curvature (cscS) Sasaki metrics. While Fano Bott manifolds are relatively sparse \cite{BoCaTo17,Suy18}, it is not so for Bott orbifolds. This then gives rise to constant scalar curvature Sasaki metrics on $S^3_\bfw$ iterated joins. However, we shall see that as the dimension grows it becomes increasingly complicated to obtain examples of cscS metrics on {\it smooth} iterated joins. Nevertheless, we certainly believe that there are many examples. We prove that there are infinitely many smooth nontrivial SE examples up through dimension 11 and conjecture it to be true in arbitrary dimension.

Since the category of Bott orbifold towers is somewhat easier to understand than the category of iterated $S^3$ joins, we begin with the former.

\begin{ack}
This paper is the basis of a planned talk by the second author at the upcoming conference 	
{\it AMAZER: Analysis of Monge-Amp\`ere, a tribute to Ahmed Zeriahi} at the Institute of Mathematics of Toulouse (June 2021\footnote{Originally the conference was planned for June 2020, but due to the Covid-19 pandemic, the conference is postponed to June 2021.}). The authors would like to take this opportunity to wish Professor Ahmed Zeriahi's  a very happy retirement. The second author would also like to thank the organizers for the invitation.
\end{ack}

\section{Orbifold Bott Towers}
As we shall see the quotient orbifolds produced by iterating the $S^3_\bfw$ join construction have the form of Bott orbifolds that we now describe. Before embarking on our journey we present a brief review of Bott manifolds.

\subsection{The Construction of Bott Towers}
We recall the definition of a Bott tower as defined by Grossberg and Karshon \cite{GrKa94} (see also \cite{BoCaTo17}): We call $M_k$ the \emph{stage $k$ Bott manifold} of the \emph{Bott tower of
  height $n$}:
\begin{equation}\label{Botttowereqn}
M_n \xrightarrow{\pi_n}M_{n-1} \xrightarrow{\pi_{n-1}} \cdots M_2
\xrightarrow{\pi_2} M_1 =\bbc\bbp^1 \xrightarrow{\pi_1} \mathit{pt}
\end{equation}
where $M_k$ is defined inductively as the total space of the projective bundle 
\begin{equation}\label{projbun}
\bbp(\BOne\oplus L_k)\xrightarrow{\pi_k}M_{k-1}
\end{equation} 
with fiber $\bbc\bbp^1$, and some holomorphic line bundle $L_k$ on $M_{k-1}$.
At each stage we have \emph{zero} and \emph{infinity sections} $\grs_k^0\colon
M_{k-1}\to M_k$ and $\grs_k^\infty\colon M_{k-1}\to M_k$ which respectively
identify $M_{k-1}$ with $\bbp(\BOne \oplus 0)$ and $\bbp(0\oplus L_k)$.
We consider these to be part of the structure of the Bott tower
$(M_k,\pi_k,\grs_k^0,\grs_k^\infty)_{k=1}^n$.

Bott towers $M_n(A)$ of dimension $n$ form the objects of a groupoid that are in one-to-one correspondence with the set of $n$ by $n$ lower triangular unipotent matrices $A=(A^j_i)$; hence, they are in one-to-one correspondence with  the set $\bbz^{\frac{n(n-1)}{2}}$. Our convention here is that $j$ labels the columns of $A$ and $i$ the rows. Note that Bott manifolds  are smooth toric varieties. A Bott manifold $M_n(A)$ has a set $\{D_{v_i},D_{u_i}\}_{i=1}^n$ of $\bbt^n$ invariant pairs of divisors  which are just the zero and infinity
sections of $\pi_i\colon \bbp(\BOne\oplus L_i)\to M_{i-1}$. That is, $D_{v_i}=\grs_i^0(M_{i-1})$ and $D_{u_i}=\grs_i^\infty (M_{i-1})$. 
respectively. The divisor classes $[D_{v_i}],[D_{u_i}]$ are elements of the Chow group $A_{n-1}(M_n(A))$ and we denote the Poincar\'e duals in $H^2(M_n(A),\bbz)$ by $y_i,x_i$, respectively. A $\bbt^n$ invariant basis of $A_{n-1}(M_n(A))$ (correspondingly $H^2(M_n(A),\bbz)$) is obtained by choosing one element from each of the invariant pairs $\{[D_{v_i}],[D_{u_i}]\}_{i=1}^n$ (correspondingly $\{y_i,x_i\}_{i=1}^n$). We note that $x_1=y_1$ (equivalently $D_{v_1}\sim D_{u_1}$), so there are at most $2^{n-1}$ invariant bases.

Let us recall the quotient construction of Bott manifolds in \cite{GrKa94,BoCaTo17}: A stage $n$ Bott manifold $M_n$ can be written as a quotient of $n$ copies of
$\bbc^2_*:=\bbc^2\setminus\{\bf0\}$ by a complex $n$-torus $(\bbc^*)^n$. To
see this, consider the action of $(t_i)_{i=1}^n\in (\bbc^*)^n$ on
$(z_j,w_j)_{j=1}^n\in (\bbc^2_*)^n$ by
\begin{equation}\label{Tnact}
(t_i)_{i=1}^n\colon (z_j,w_j)_{j=1}^n\mapsto
\Bigl(t_jz_j,\Bigl(\prod_{i=1}^{n}t_i^{A^i_j}\Bigr) w_j\Bigr)\,\strut_{j=1}^n,
\end{equation}
where $A$ is a lower triangular unipotent integer-valued matrix
\begin{equation}\label{Amatrix}
A=\begin{pmatrix}
1   & 0 &\cdots  & 0 &0 \\
A^1_2& 1    &\cdots &0  &0\\
\vdots&\vdots&\ddots&\vdots& \vdots\\
A^1_{n-1}& A^2_{n-1}& \cdots& 1 &0\\
A^1_n   & A^2_n & \cdots & A^{n-1}_n &  1
\end{pmatrix}, \qquad A^i_j\in\bbz.
\end{equation}
Since the induced action of $(\bbr^+)^n$ is transverse to $(S^3)^n$, where
$S^3$ is the unit sphere in $\bbc^2_*$, the orbits of this action are in
bijection with orbits of the induced free $(S^1)^n$ action on $(S^3)^n$, and
the geometric quotient is a compact complex $n$-manifold $M_n(A)$.

\subsection{The Category of Bott Orbifold Towers}
We employ the notation of \cite{BoTo19a} by writing $D_{u_i^0},D_{u_i^\infty}$ instead of $D_{v_i},D_{u_i}$, and construct Bott orbifold towers simply by putting an orbifold structure on the invariant divisors $D_{u_i^0},D_{u_i^\infty}$ to make them branch divisors with ramification indices $m_i^0,m_i^\infty$, respectively. 
\begin{definition}\label{Bottorbtow}
A {\it Bott orbifold tower of  height $n$} is a Bott tower of orbifolds of the form
\begin{equation}\label{Bottorbtowereqn}
(M_n(A),\grD^T_{\bfm_n}) \xrightarrow{\pi_n}(M_{n-1}(A),\grD^T_{\bfm_{n-1}}) \xrightarrow{\pi_{n-1}} \cdots (M_2(A),\grD^T_{\bfm_2})
\xrightarrow{\pi_2} (M_1 =\bbc\bbp^1,\grD_{\bfm_1}) \xrightarrow{\pi_1} (\mathit{pt},\emptyset)
\end{equation}
where Bott orbifold towers are log pairs $(M_n(A),\grD^T_\bfm)$ with
\begin{equation}\label{Bottbranchdiv}
\grD^T_{\bfm_n}=\sum_{i=1}^n\grD_{\bfm_i}=\sum_{i=1}^n\Bigl(\bigl(1-\frac{1}{m_i^0}\bigr)D_{u^0_i} +\bigl(1-\frac{1}{m_i^\infty}\bigr)D_{u^\infty_i}\Bigr)
\end{equation} 
and $(M_i,\grD^T_i)$ is the log pair associated to the total space of the projective orbibundle 
$$\bbp_\bfm(\BOne\oplus L_i)\xrightarrow{\pi_i}(M_{i-1},\grD^T_{\bfm_{i-1}})$$ 
with fiber $\bbc\bbp^1(v_i^0,v_i^\infty)/\bbz_{m_i}$ for some holomorphic line orbibundle $L_i$ on $M_{i-1}$ where $\bfm_i=(m_i^0,m_i^\infty)=m_i(v_i^0,v_i^\infty)$ with $v_i^0,v_i^\infty$ relatively prime, so $m_i=\gcd(m_i^0,m_i^\infty)$. 
The set of Bott orbifold towers of height $n$ forms a groupoid whose objects $\calb\calo^n_0$ are Bott orbifold towers $(M_n(A),\grD_{\bfm_n}^T)$ and whose morphisms $\calb\calo^n_1$ are orbifold biholomorphisms. The quotient stack $\calb\calo^n_0/\calb\calo^n_1$ is bijective to the set of biholomorphism classes of Bott orbifolds $\{(M_n(A),\grD_{\bfm_n}^T)\}$. The {\it orbit of $\calb\calo^n_1$ through $(M_n(A),\grD_{\bfm_n}^T)$} is the set of all Bott tower orbifolds that are equivalent to $(M_n(A),\grD_{\bfm_n}^T)$. 
\end{definition}

All categories in this paper are small. The equivalences that are of most interest to us in this paper are those  induced by the fiber inversion maps as described in \cite{BoCaTo17,BoTo19a}. Such Bott tower orbifolds take the form $(M_n(A'),\grD_{\bfm^\perp_n}^T)$ where $A'$ is described in Section 1.4 of \cite{BoCaTo17}. Since the entire singular orbifold stratum has complex codimension one, it follows that as an algebraic variety $(M_n(A),\grD^T_\bfm)$ is biholomorphic to the Bott manifold $M_n(A)$. Notice also that, since $M_0$ is a point, in the case $i=1$ with 
$$\grD_{\bfm_1}=\bigl(1-\frac{1}{m_1^0}\bigr)D_{u^0_1} +\bigl(1-\frac{1}{m_1^\infty}\bigr)D_{u^\infty_1},$$
we can identify $(M_1,\grD_{\bfm_1})$ with a quotient of the weighted projective line, viz. $\bbc\bbp^1[v_1^0,v_1^\infty]/\bbz_{m_1}$.
For ease of notation we will often write this as $\bbc\bbp^1[\bfm_1]$. The case of the Bott manifolds studied in \cite{BoCaTo17} occurs when $\bfm_i=(m_i^0,m_i^\infty)=(1,1)$ for all $i$ and is denoted by $(M_i(A),\emptyset)$ or simply $M_i(A)$ when it is clear from the context. 

We let $\calb\calo_0$ denote the set of all Bott orbifold towers with morphisms $\calb\calo_1$ consisting of orbifold maps that preserve the tower structure \eqref{Bottorbtowereqn}. These are the equivalences in the groupoid $\calb\calo^k$, the projection maps $\pi_k$ in the tower, and the section maps $\grs^0_k,\grs^\infty_k$. This forms a special type of category called a {\it topos}. For the basics of topos theory we refer to the books \cite{Joh77,Gol84,MaMo94}:

\begin{lemma}\label{powob}
The category $\calb\calo$ is a topos which at each stage $k$ is the groupoid $\calb\calo^k$.
\end{lemma}

\begin{proof}
We outline a proof which can be checked as follows. First, from the sequence \eqref{Bottorbtowereqn} we see that the singleton $(pt,\emptyset)$ is a terminal object. Second, at each stage diagrams pullback. Finally, the section maps $\grs^0,\grs^\infty$ generate the subobjects which are the subsets giving power objects. So $\calb\calo$ is finitely complete and has power objects; hence, it is a topos [\cite{Gol84}, pg 106]. That each stage is a groupoid follows by construction, and the morphisms $\pi_k,\grs^0_k,\grs^\infty_k$ are functors of the groupoids.
\end{proof}

Since orbifolds can be represented by proper \'etale Lie groupoids, $\calb\calo$ is a category whose objects are categories. Moreover, a morphism of $\calb\calo$, that is orbifold diffeomorphisms, corresponds to Morita equivalence of the groupoids \cite{MoPr97,Moe02,MoMr03,BG05}. So the catagory of Bott orbifolds is a 2-category \cite{Ler10}.

\begin{remark}\label{catrem}
The category $\calb\calo$ has both inverse and direct limits. The former is given by the inverse limit of the sequence \eqref{Bottorbtowereqn}; whereas, the latter is the direct limit of the system 
$$(\{pt\},\emptyset)\fract{\grs_1^i}{\longrightarrow} (\bbc\bbp^1,\grD_{\bfm_1})\fract{\grs_2^i}{\longrightarrow} \cdots \longrightarrow(M_{k-1},\grD^T_{\bfm_{k-1}}) \fract{\grs_k^i}{\longrightarrow} (M_{k},\grD^T_{\bfm_{k}})\longrightarrow\cdots  $$
given by choosing $i=0,\infty$ at each stage. The singleton $(\{pt\},\emptyset)$ is both a terminal and initial object, so the category $\calb\calo$ has both limits and colimits. There are natural topologies\footnote{Grothendieck topologies are not topologies in the usual sense. They deal with `open coverings' which are not necessarily open sets.} on topoi known as {\it Grothendieck topologies}, but we do not make explicit use of them here. Suffice it to say that $\calb\calo$ has an appropriate Grothendieck topology.
\end{remark}



\subsection{The Orbifold First Chern Class}
We need to compute the orbifold first Chern class with respect to each of the $\bbt^n$ invariant bases. Let us begin by making note of the relations
\begin{equation}\label{yxrel}
y_i=\sum_{j=1}^n A^j_ix_j=x_i+\sum_{j=1}^{i-1}A^j_ix_j,\qquad x_i=\sum_{j=1}^n (A^{-1})^j_iy_j=y_i+\sum_{j=1}^{i-1}(A^{-1})^j_iy_j .
\end{equation}
Then we compute the orbifold first Chern class in the basis $\{x_j\}$, 
\begin{eqnarray}\label{c1orb}
c_1^{orb}(M_\bfn(A),\grD^T_\bfm) &=&c_1(M_\bfn(A)) -\sum_{i=1}^n\Bigl(\bigl(1-\frac{1}{m_i^0}\bigr)y_i +\bigl(1-\frac{1}{m_i^\infty}\bigr)x_i\Bigr) \notag \\
                                             &=& \sum_{i=1}^n(x_i+y_i)-\sum_{i=1}^n\Bigl(\bigl(1-\frac{1}{m_i^0}\bigr)y_i +\bigl(1-\frac{1}{m_i^\infty}\bigr)x_i\Bigr) \notag \\
                              &=&\sum_{j=1}^n\bigl(\frac{1}{m_j^0}y_j+\frac{1}{m_j^\infty}x_j\bigr)  =\sum_{i=1}^n\Bigl((\frac{1}{m_i^0}+\frac{1}{m_i^\infty})x_i+\sum_{j=1}^{i-1}\frac{A^j_i}{m_i^0}x_j\Bigr) \notag \\
                              &=& \sum_{j=1}^{n-1}\Bigl(\frac{1}{m_j^0}+\frac{1}{m_j^\infty}+\sum_{i=j+1}^n\frac{A^j_i}{m_i^0}\Bigr)x_j +\Bigl(\frac{1}{m_n^0}+\frac{1}{m_n^\infty}\Bigr)x_n.
\end{eqnarray}

\begin{lemma}\label{c1it}
On Bott orbifolds, $c_1^{orb}$ satisfies the following recursion relation
$$c_1^{orb}(M_n(A),\grD^T_{\bfm_{n}}) = c_1^{orb}(M_{n-1}(A),\grD^T_{\bfm_{n-1}}) +\sum_{j=1}^{n-1}\frac{A^j_n}{m_n^0}x_j +(\frac{1}{m^0_n}+\frac{1}{m^\infty_n})x_n$$
where the matrix $A$ in $M_{n-1}(A)$ is $A$ with the nth row and column deleted. Moreover, if $c_1^{orb}(M_n(A),\grD^T_{\bfm_{n}})$ is log Fano then $c_1^{orb}(M_{n-1}(A),\grD^T_{\bfm_{n-1}})$ is log Fano.
\end{lemma}

\begin{proof}
The recursion formula follows easily from \eqref{c1orb}, while the second statement follows from Proposition 4.2 of \cite{BoTo15}.
\end{proof}

\begin{remark}
It is easy to see by examples that the converse of the second statement in Lemma \ref{c1it} does not hold generally.
\end{remark}

The form of the first Chern class \eqref{c1orb} of the orbifold canonical divisors suggest that we extend these rational cohomology classes to the entire K\"ahler cone $\calk$ and deal with real cohomology classes as well as real divisors. Thus, it is convenient to introduce $\{(q_j^0,q_j^\infty)=(\frac{1}{m_j^0},\frac{1}{m_j^\infty})\}_{j=1}^n$ and consider them as  coordinates for $(\bbr^+)^{2n}$. So we rewrite \eqref{c1orb} as
\begin{equation}\label{c1orb2}
c_1^{orb}(M_\bfn(A),\grD^T_\bfm)= \sum_{j=1}^{n-1}\Bigl(q_j^0+q_j^\infty+\sum_{i=j+1}^nA^j_iq^0_i\Bigr)x_j +\Bigl(q_n^0+q_n^\infty\Bigr)x_n.
\end{equation}
Let $G\approx (\bbz_2)^{n}$ denote the group generated by the fiber inversion maps $\grt_j$ which interchange the zero and infinity sections $\grs_i^0,\grs_i^\infty$ and hence interchange the invariant divisors $D_{u_i^0}$ and $D_{u_i^\infty}$. This induces an action interchanging the divisors classes $[D_{u_i^0}]$ and $[D_{u_i^\infty}]$ as well as their duals $x_i$ and $y_i$. So the set of $\bbt^n$ invariant bases is a $G$-module. Fixing a basis, say $\{x_i\}$ of $\bbt^n$ invariant divisors, we obtain an arbitrary $\bbt^n$ invariant basis by an action of $G$. So the set of all such bases is given by $\{g(x_i)\}_{g\in G}$. This induces an action of $G$ on the set $UP$ of lower triangular unipotent matrices. From this we obtain an action of $G$ on $\calb\calo_0^n=\bbz^{\frac{n(n-1)}{2}}\times(\bbz^+)^{2n}$. Note that the action $\grt_j\in G$ on $(\bbz^+)^{2n}=(\bbz^+)^{n}\times (\bbz^+)^{n}$ is the interchange map on the two jth factors; it is free away from the diagonal in the $(\bbz^+)^{2n}_j=(\bbz^+)^{n}_j\times (\bbz^+)^{n}_j$.
Moreover, this action extends to an action on $UP\times(\bbr^+)^{2n}\approx \bbz^{\frac{n(n-1)}{2}}\times(\bbr^+)^{2n}$. Note that $\grt_1$ acts as the identity on $UP$ as well as on the invariant divisor classes and their dual. Thus, it acts as the identity on $UP\times D_1\times(\bbr^+)^{2n-2}$ where $D_1$ is the diagonal in the first factor $\bbr\times\bbr$.

\begin{proposition}\label{c1orbprop}
For each $g\in G=(\bbz_2)^{n}$ the orbifold first Chern class of the Bott orbifold $(M_n(A),\grD_\bfm)$ takes the form 
\begin{equation}\label{c1orbgen}
c_1^{orb}(M_\bfn(A),\grD^T_\bfm)=\sum_{j=1}^{n-1}\Bigl(q_j^0+q_j^\infty+\sum_{i=j+1}^ng(A^j_iq^0_i)\Bigr)g(x_j) +\Bigl(q^0_n+q^\infty_n\Bigr)g(x_n) 
\end{equation}
where $g(A^j_iq^0_i)$ has the form
\begin{equation}\label{Beqn}
g(A^j_iq_i^0)= \begin{cases} B^j_iq_i^0 &~~\text{if $g(x_i)=x_i$}, \\
                                    B^j_iq_i^\infty &~~\text{if  $g(x_i)=y_i$,}
                 \end{cases} 
\end{equation}
and $B^j_i$ is determined by $g$ and $A$.
\end{proposition}

Changing from the basis element $x_i$ to $y_i$ amounts to applying the fiber inversion map $\grt_i$ interchanging $D_{u^0_i}$ and $D_{u^\infty_i}$. Thus, to write $c_1^{orb}$ in an arbitrary invariant basis, we need to apply a product of fiber inversion maps $\grt_{k_1}\cdots\grt_{k_r}\in G$. Generally, the $\grt_k$ are morphisms of the Bott tower groupoid described in Lemma 1.10 of \cite{BoCaTo17}. As explained in the proof of that lemma, the kth fiber inversion map $\grt_k$ is induced by the equivalence which interchanges the kth column of the matrix $-A$ with the kth column of the identity matrix $\BOne$ in $(-A\quad \BOne)$. Putting the image in normal form gives the equivalent matrix $\grt_k(A)$. Thus, from any product $g=\grt_{k_1}\cdots\grt_{k_r}$ of fiber inversion maps one obtains an equivalent Bott tower with a lower triangular unipotent matrix $g(A)$. This gives an affine map on $(\bbz^+)^{2n}$ which extends to an affine map on $(\bbr^+)^{2n}$. Moreover, since $g(A)$ is equivalent to $A$, $G$ is a subgroup of the isotropy group at $A$ in $\calb\calo_1$, namely, the group $\gA\gu\gt((M_n(A),\grD^T_\bfm)$ of automorphisms of the Bott orbifold $(M_n(A),\grD^T_\bfm)$.

\subsection{The Log Fano Condition}
We are interested in the monotone case with positive orbifold first Chern class, that is when the orbifolds $(M_n(A),\grD^T_\bfm)$ are log Fano. We have the following analog of Lemma 3.10 of \cite{BoCaTo17}
\begin{lemma}\label{amplelem}
A $\bbt^c$-invariant $\bbr$ divisor $D$ is ample if
and only if in every $\bbt^c$-invariant basis $[D_1],\ldots, [D_n]$ of
$A_{n-1}(M_n(A))$ with $D_1=D_{u^0_1}$ and either $D_i=D_{u^0_i}$ or
$D_i=D_{u^\infty_i}$, the coefficients $r_i$ of $[D]=\sum_{i=1}^n r_i [D_i]$ are all positive.
\end{lemma}

\begin{proof}
Using the Nakai Criterion for $\bbr$ divisors (cf. Theorem 2.3.18 of \cite{Laz04a}) the proof proceeds as in that of Lemma 3.10 of  \cite{BoCaTo17}.
\end{proof}

Lemma \ref{amplelem} implies that $(M_n(A),\grD^T_\bfm)$ is log Fano if and only if $c_1^{orb}(M_n(A),\grD^T_\bfm)$ has positive coefficients with
respect to each of the $2^{n-1}$ invariant bases of $H^2(M_n,\bbz)$. When these conditions are satisfied $(M_n(A),\grD^T_\bfm)$ will be a monotone log Fano orbifold, and we can search for orbifold K\"ahler Ricci solitons and K\"ahler-Einstein metrics.

We have an immediate consequence of Proposition \eqref{c1orbprop}
\begin{lemma}\label{logFanolem}
$(M_n(A),\grD^T_\bfm)$ is log Fano if and only if for all $j=1,\ldots,n-1$ the inequality
$$\frac{1}{m_j^0}+\frac{1}{m_j^\infty}+\sum_{i=j+1}^nB^j_i>0$$
holds, where $B^j_i$ is determined in Proposition \eqref{c1orbprop}.
\end{lemma}

Note that these positivity conditions imply that $c_1^{orb}(M_\bfn(A),\grD^T_\bfm)$ lies in the K\"ahler cone $\calk$ of the Bott manifold $M_n(A)$. In fact, if we allow $m_i^0,m_i^\infty$ to take on values in the positive real numbers, they parameterize the entire K\"ahler cone. Geometrically, when we allow $m_i^0,m_i^\infty$ to be any positive real numbers we are dealing with $M_n(A)$ with cone singularities along the $\bbt$ invariant divisors $D_{u^0_i},D_{u^\infty_i}$ with cone angle $\frac{2\pi}{m_i^0}$ and  $\frac{2\pi}{m_i^\infty}$, respectively. See \cite{Don12} for this approach. Summarizing we have





\begin{proposition}\label{Kconeprop}
Let $\gro$ be an orbifold symplectic form. Then $[\gro]$ lies in the K\"ahler cone $\calk(M_n(A))$ if and only if its coefficients with respect to each of the $2^n$ bases of invariant forms obtained by choosing one from each invariant pair $\{x_i,y_i\}_{i=1}^n$ are positive. 
\end{proposition}

\section{The Iterated $S^3$-Join Construction}
As most recently described in \cite{BoTo19a} the $S^3_\bfw$-join 
$$M_{\bfl,\bfw}=M\star_\bfl S^3_\bfw$$ 
where $M$ is a Sasaki manifold (orbifold) with a quasi-regular Reeb vector field $\xi_M$ is then constructed from the following  commutative diagram
\begin{equation}\label{s2comdia}
\begin{matrix}  M\times S^3_\bfw &&& \\
                          &\searrow\pi_L && \\
                          \decdnar{\pi_{2}} && M_{\bfl,\bfw} &\\
                          &\swarrow\pi_1 && \\
                           (N,\grD_N)\times\bbc\bbp^1[\bfw] &&& 
\end{matrix} 
\end{equation}
where the $\pi_2$ is the product of the projections of the standard Sasakian projections $\pi_M:M\ra{1.6}  (N,\grD_N)$ and $S^3_\bfw\ra{1.6} \bbc\bbp^1[\bfw]$. The circle action on $M\times S^3$ is generated by the vector field 
\begin{equation}\label{Lvec}
L_{\bfl,\bfw}=\frac{1}{2l^0}\xi_M-\frac{1}{2l^\infty}\xi_\bfw, 
\end{equation}
where $\bfl=(l^0,l^\infty)$ has relatively prime components, and $\xi_\bfw$ is the quasiregular weighted Reeb field with weights $\bfw=(w^0,w^\infty)$ on $S^3_\bfw$. The join $M_{\bfl,\bfw}$ has has a naturally induced quasi-regular Sasakian structure $\cals_{\bfl,\bfw}$ with contact 1-form $\eta_{\bfl,\bfw}$ and Reeb vector field
\begin{equation}\label{Reebjoin}
\xi_{\bfl,\bfw}=\frac{1}{2l^0}\xi_M+\frac{1}{2l^\infty}\xi_\bfw.
\end{equation}
The contact bundle splits as
$$\cald_{\bfl,\bfw}=\cald_M\oplus \cald_3$$
where $\cald_M$ is the contact bundle of $M$ and $\cald_3$ is the contact bundle of $S^3$, and the subbundle $\cale_M=\cald_M\oplus L_{\xi_\bfw}$ is integrable. The leaves are copies of $M$ with Sasakian structure whose Reeb vector field is $\xi_M$ up to transverse scaling.

We now specialize to the case of $S^3_\bfw$-iterated joins. Our first result shows that in this case the quotients of any quasiregular Reeb vector field in $\gt^+_\bfw$ is a Bott orbifold. Hence, the orbifold singularities are all of codimension one over $\bbc$. We consider the iterated $S^3$-join construction and we refer to Sections 2.3 and 2.4 of \cite{BHLT16} for its description.

We do not want to restrict ourselves to the completely reducible case, that is where $\cald$ splits as $\oplus \cald_i$ with $\dim_\bbr \cald_i=2$, but only to the `cone reducible' case which is given by a filtration of $\cald$ with 2-dimensional quotients. This is needed to be able to effectively implement the admissibility conditions of the generalized Calabi construction. Explicitly, we need to be able to choose a quasiregular Sasaki metric at each stage of the iteration. It is convenient to change notation a bit. We have weight vectors $\bfw_i=(w_i^0,w_i^\infty)$ with $i=0,\ldots,k$ and $\bfl_j=(l_j^0,l_j^\infty)$ with $j=1,\ldots,k$, so $\bfl,\bfw$ are multi-indices. For simplicity we restrict ourselves to the case where the components of $\bfw_i$ and $\bfl_j$ are relatively prime. We consider the $S^3$-iterated join of height $k$ 
\begin{equation}\label{itjoin}
M^{2k+1}_{\bfl,\bfw}=M^{2k-1}_{\bfl,\bfw}\star_{\bfl_{k-1}}S^3_{\bfw_{k}}=\bigl(\bigl(S^3_{\bfw_1}\star_{\bfl_1}S^3_{\bfw_2}\bigr)\star_{\bfl_2}\cdots \star_{\bfl_{k-2}}S^3_{\bfw_{k-1}}\bigr) \star_{\bfl_{k-1}}S^3_{\bfw_{k}}.
\end{equation}
as an element of $\calo_{2k+1}$ as described in \cite{BGO06,BG05,BHLT16,BoTo19a}. The orbifolds $M^{2k+1}_{\bfl,\bfw}$ are compact of dimension $2k+1$. In our iterated join \eqref{itjoin} we began with the weighted 3-sphere $S^3_{\bfw_1}$; however, it is convenient to begin with the 1-dimensional Sasaki manifold, namely, $S^1$ with its flat metric. Then by choosing $\bfl_0=(1,1)$ we obtain the weighted 3-sphere as the join $S^1\star_{1,1,}S^3_\bfw=S^3_\bfw$. Then at each stage $k$ we choose a quasi-regular Sasakian structure given by a primitive quasi-regular Reeb field $\xi_{\bfv_j}$ in the subcone $\gt^+_{\bfw_k}$. Since the join operation is non-associative we fix the order as is indicated in Equation \eqref{itjoin}. This can then be written as a $(k-1)$ dimensional torus action on the k-fold product $S^3\times\cdots\times S^3$ as follows. First we represent the jth sphere $S^3$ in terms of complex coordinates $(z_j^0,z_j^\infty)$ satisfying $|z_j^0|^2+|z_j^\infty|^2=1$ for $j=1,\ldots,k$. Then the $\bbt^{k-1}$ action on the jth sphere for $j=2,\ldots,k-1$ is 
\begin{equation}\label{Tkactsph}
(z_j^0,z_j^\infty)\mapsto (e^{i(m_jv_j^0\theta_{j}-l_{j-1}^0w_j^0\theta_{j-1})}z_j^0, e^{i(m_jv_j^\infty\theta_{j}-l_{j-1}^0w_j^\infty\theta_{j-1})}z_j^\infty)
\end{equation}
whereas, for $j=1$ we have 
\begin{equation}\label{Tk0}
(z_1^0,z_1^\infty)\mapsto (e^{il_1^\infty
w_1^0\theta_1}z_1^0,e^{il_1^\infty w_1^\infty\theta_1}z_1^\infty),
\end{equation} 
and for $j=k$
\begin{equation}\label{Tkk}
(z_k^0,z_k^\infty)\mapsto (e^{-il_{k-1}^0w_k^0\theta_{k-1}}z_k^0,e^{-il_{k-1}^0w_k^\infty\theta_{k-1}}z_k^\infty)
\end{equation}
with the relations $l^\infty_{j-1}=m_js_j$ and $s_j=\gcd(l^\infty_{j-1},|w^0_jv^\infty_j-w^\infty_j v^0_j|)$.
Note that $\theta_j$ occurs only in the jth and (j+1)st $S^3$ for $j=2,\ldots,k-1$. This action gives the torus bundle
\begin{equation}\label{Tkact}
\bbt^{k-1}\ra{2.8} S^3_{\bfw_1}\times\cdots\times S^3_{\bfw_{k}}\ra{2.8} M_{\bfl,\bfw}^{2k+1}.
\end{equation}

\subsection{The Category of Iterated $S^3_\bfw$ Joins} 
We let $\cals\calj_0$ denote the set of all $S^3_\bfw$ iterated Sasaki joins beginning with the 1 dimensional Sasaki manifold described by a circle $S^1$ with the flat metric. 
This $S^1$ provides the set $\cals\calj_0$ with both a limit and a colimit. Thus, we have towers of the form
\begin{equation}\label{itjointow}
\cdots\longrightarrow M^{2k+1}_{\bfl,\bfw}\longrightarrow   \cdots\longrightarrow M^5_{\bfl,\bfw} \longrightarrow M^3_{\bfl,\bfw}\longrightarrow M^1=S^1
\end{equation}
and
\begin{equation}\label{itjointow2}
S^1=M^1\longrightarrow M^3_{\bfl,\bfw}\longrightarrow M^5_{\bfl,\bfw} \longrightarrow \cdots\longrightarrow M^{2k+1}_{\bfl,\bfw}\longrightarrow   \cdots ~.
\end{equation}
At each stage $k$ (except possibly the last) we choose a quasi-regular Sasakian structure in the $\bfw_k$ subcone of the Sasaki cone. That this can always be done at least in the orbifold category follows inductively from \cite{BoTo19a}. So at stage $k$ the set of objects $\cals\calj_0^k$ consists of Sasaki orbifolds $\calj_{\bfl,\bfw,\bfv}$ that are $k$-fold $S^3_\bfw$ iterated joins of the form \eqref{itjoin} such that at each stage $i=1,\ldots,k-1$, (except possibly the last stage $k$) a quasi-regular Reeb vector field $\xi_{\bfv_i}\in \gt^+_{\bfw_i}$ is chosen where $\bfl$, $\bfw$ and $\bfv$ are multi-indices such that for each $i$ the components $(l^0_i,l^\infty_i)$, $(w^0_i,w^\infty_i)$ and $(v_i^0,v^\infty_i)$ are relatively prime positive integers\footnote{For simplicity we take the components $l^0_i,l^\infty_i$ and $w^0_i,w^\infty_i$ to be relatively prime for each $i$. This implies that the orbifolds $M_{\bfl,\bfw}$ are simply connected. Dropping this condition defines a larger groupoid $\tilde{\calj}$ such that $\calj\subset \tilde{\calj}$ is full subgroupoid.}. A subobject in $\cals\calj_0$ is the Sasakian structure represented by the leaves of the foliation $\cale_i=\cald_i+L_{\xi_{\bfv_i}}$ of the $i$th stage of the iterated join. The morphisms $\cals\calj_1$ consist of the restrictions to the subobjects, the inclusion of the subobjects, and orbifold diffemorphisms that intertwine the Sasakian structures at each stage. The latter are the morphisms of the groupoid $\cals\calj^k$ of equivalences defined by orbifold diffeomorphisms $f:M_{\bfl,\bfw}\longrightarrow M_{\bfl',\bfw'}$ 
such that 
\begin{equation}\label{Deqn}
f_*\cald_{\bfl,\bfw}=\cald_{\bfl',\bfw'},\qquad f_*\circ J_\bfw=J_{\bfw'}\circ f_*
\end{equation}
and
$f^*\cals_{\bfl',\bfw',\bfv'}=\cals_{\bfl,\bfw,\bfv}$. Note that the isotropy subgroup of $\cals\calj_1$ at $\cals_{\bfl,\bfw,\bfv}$ is just the Sasaki automorphism group $\gA\gu\gt(\cals_{\bfl,\bfw,\bfv})$. The quotient stack $\cals\calj_0/\cals\calj_1$ then consists of the orbifold diffeomorphism classes of iterated $S^3_\bfw$ Sasaki joins. This determines the underlying CR structure $(\cald_{\bfl,\bfw},J_\bfw)$ and equivalently the transverse complex structure. So the choice of Reeb vector fields $\xi_{\bfv}=(\xi_{\bfv_1},\ldots,\xi_{\bfv_k})$ gives the object set $\cals\calj_0^k$ consisting of all such Sasakian structures 
$$\cals\calj_0=\{\calj_{\bfl,\bfw,\bfv}\}=\{\cals_{\bfl,\bfw,\bfv}\}_{(\bfl,\bfw,\bfv)\in(\bbz^+)^2\times(\bbz^+)^2\times (\bbz^+)^2}$$
with relatively prime components for each $i$ as stated above. As in the case of Bott orbifolds we have

\begin{proposition}\label{itjointopos}
$\cals\calj$ is a topos of groupoids.
\end{proposition}



While the components $l^0_i,l^\infty_i$ and $(w^0_i,w^\infty_i)$ being relatively prime implies that all the objects are associated with simply connected Sasaki orbifolds, they are generally topologically different which we describe in more detail below. At stage 2 there are precisely two diffeomorphism types determined by the Stiefel-Whitney class of the contact bundle $\cald_{\bfl,\bfw}$. To each element $\cals\in\calj_{\bfl,\bfw,\bfv}$ we can associate the first Chern class $c_1(\cald_{\bfl,\bfw})$ of the contact bundle. In fact $c_1(\cald_{\bfl,\bfw})$ is invariant under equivalence, and so only depends on the class in $\calj_0/\calj_1$. Of course, the mod 2 reduction of $c_1(\cald_{\bfl,\bfw})$ is the Stiefel-Whitney class. However, as we shall see, at higher stages, there are infinitely many diffeomorphism types which are distinguished by torsion classes.

\subsection{The Topology of the $S^3$-Iterated Joins}
The (orbifold) order of a quasi-regular Sasakian structure $\cals_k$ with Reeb field $\xi_\bfv$ on the height $k$ iterated join \eqref{itjoin} is given by 
\begin{equation}\label{sasorder}
\Upsilon_{\cals_k}=m_k v_k^0v_k^\infty\cdots m_2v_2^0v_2^\infty w_1^0w_1^\infty=\prod_{i=2}^k\lcm(m_i^0,m_i^\infty)w_1^0w_1^\infty
\end{equation}
where
$$m_i=\frac{l^\infty_{i-1}}{\gcd(l^\infty_{i-1},|w^0_iv^\infty_i-w^\infty_i v^0_i|)}.$$
So from Lemma 2.3 of \cite{BoTo19a} we have
\begin{lemma}\label{smoothlem}
The iterated join \eqref{itjoin} is smooth if and only if 
$$\gcd(l_{k-1}^\infty\Upsilon_{\cals_{k-1}},l_{k-1}^0 w_k^0w_k^\infty)=1.$$
\end{lemma}

First we note that the homotopy groups can be determined from the homotopy groups of spheres by the long exact homotopy sequence of the torus bundle \eqref{Tkact}; however, for topological computations it is more convenient to write this fibration as
\begin{equation}\label{Tkact2} 
S^3_{\bfw_1}\times\cdots\times S^3_{\bfw_{k}}\ra{2.8} M_{\bfl,\bfw}^{2k+1}\ra{2.8} \mathsf{B}\bbt^{k-1} .
\end{equation}
In particular we have 
$$\pi_1(M_{\bfl,\bfw}^{2k+1})=\{\BOne\},~\pi_2(M_{\bfl,\bfw}^{2k+1})=\bbz^{k-1},~\pi_3(M_{\bfl,\bfw}^{2k+1})=\bbz^{k},~\pi_4(M_{\bfl,\bfw}^{2k+1})=\bbz_2^{k},~\cdots.$$
So as in Lemma 3.2 of \cite{BoTo18c} we have
\begin{lemma}\label{exhomseq}
The (2k+1)-manifolds $M_{\bfl,\bfw}^{2k+1}$ satisfy the following conditions: 
\begin{enumerate}
\item $H_1(M^{2k+1}_{\bfl,\bfw},\bbz)=\pi_1(M_{\bfl,\bfw}^{2k+1})=\{\BOne\}$,
\item $\pi_2(M_{\bfl,\bfw}^{2k+1})=\bbz^{k-1},~\pi_3(M_{\bfl,\bfw}^{2k+1})=\bbz^{k}$,
\item $H^2(M_{\bfl,\bfw}^{2k+1},\bbz)=H_2(M_{\bfl,\bfw}^{2k+1},\bbz)=\bbz^{k-1}$,
\item $H^3(M_{\bfl,\bfw}^{2k+1},\bbz)$ is torsion free.
\item The Betti numbers $b_3(M_{\bfl,\bfw}^{2k+1}),~\cdots b_{2\lfloor\frac{k+2}{2}\rfloor-1}(M_{\bfl,\bfw}^{2k+1})$ and their Poincar\'e duals are even
where $\lfloor\cdot\rfloor$ is the floor function. 
\end{enumerate}
\end{lemma}

\begin{lemma}\label{H3H4}
For the iterated join $M^{2k+1}_{\bfl,\bfw}$ with $k\geq 3$ the following hold:
\begin{enumerate}
\item $H^3(M^{2k+1}_{\bfl,\bfw},\bbz)=0$;
\item the free part of $H^4(M^{2k+1}_{\bfl,\bfw},\bbz)$ is $\bbz^{\frac{k(k-3)}{2}}$.
\end{enumerate}
\end{lemma}

\begin{proof}
Since $H^3(M^{2k+1}_{\bfl,\bfw},\bbz)$ is torsion free, it suffices to work over $\bbq$. Consider the Leray-Serre cohomology spectral sequence of the bundle $S^1\ra{1.5}M^{2k+1}\ra{1.5}M_{k}(A)$. The only 3 dimensional classes in the $E^2$ term take the form $\gra\otimes u$ where $\gra$ is the 1 dimensional class of the fiber and $u$ a 2 dimensional class on the base. If such a 3 dimensional class were to survive to $E^\infty$ we would have $d_2(\gra\otimes u)=0$. But $d_2(\gra)$ is the K\"ahler class so $d_2(\gra)=\sum_{j=1}^{k}c_jx_j$ where $\{x_j\}$ is a basis for $H^2(M_{k}(A),\bbq)$ and $c_j>0$. Now the cohomology ring of the Bott manifold $M_{k}(A)$ is the polynomial ring $\bbq[x_1,\ldots,x_{k}]$ modulo the ideal generated by $x_j^2+\sum_{i=1}^{j-1}A^i_jx_ix_j$.
Using this we see that $d_2(\gra\otimes u)=0$ gives the system of equations
$$c_ju_i+c_iu_j-A^i_jc_ju_j=0$$
which can be written in matrix form $Cu=0$ where $C$ is a $k$ by $k(k-1)/2$ matrix. Since $c_j>0$ for all $j$ it is not difficult to see that $C$ has at least $k-1$ independent columns. So the rank of $C$ is either $k$ or $k-1$. If it is $k$ we get $u=0$, and if it is $k-1$ there is precisely one non-zero solution which contradicts the fact that $b_3$ is even. This proves (1).	

To prove (2) we consider the Leray-Serre spectral sequence of the fibration \eqref{Tkact2}. Now $H^4(\mathsf{B}\bbt^{k-1},\bbz)=\bbz^{\frac{k(k-1)}{2}}$, and by (1) none of the 3 dimensional classes on the fiber survive to $E^\infty$,  so we must have $\frac{k(k-1)}{2}-k=\frac{k(k-3)}{2}$ 4 dimensional  free classes that survive to $E^\infty$.
\end{proof}

We also correct a small error in Theorem 3.1 of \cite{BoTo18c}, namely that the torsion group $\bbz_{l_2^2m^2v^0v^\infty}$ of $M^7$ should be $\bbz_{l_2m^2v^0v^\infty}$, that is $l_2$ should only occur to the first power. The error was that it is the $d_2$ differential of the 3-class $\grb$, not the $d_4$ differential that occurs in the $E_2$ term of the spectral sequence. We have $d_2(\grb)=m^2v^0v^\infty s_1$. The same argument clearly holds for the general 7-dimensional  cscS case, namely we have

\begin{theorem}\label{maintopthm}
The 7-manifolds $M^7=(S^3\star_{\bfl_1}S^3_{\bfw_1})\star_{\bfl_2}S^3_{\bfw_2}$ have the rational cohomology of the connected sum $(S^2\times S^5)\# (S^2\times S^5)$. Furthermore, the only torsion that occurs is $H^4(M^7,\bbz)\approx \bbz_{v^0v^\infty m^2l_2^\infty}\oplus \bbz_{w_2^0w_2^\infty (l_2^0)^2}$ where $\xi_\bfv$ with $\bfv=(v^0,v^\infty)$ is the Reeb vector field of the quasi-regular cscS metric on $S^3\star_{\bfl_1}S^3_{\bfw_1}$.
\end{theorem}

\section{Iterated Joins and K\"ahler Bott Orbifolds}

\subsection{The Orbifold Quotients} 
We want to understand the quotient orbifolds of a quasi-regular Sasakian structure in the $\bfw$ Sasaki cone of an iterated join. To do so we consider the objects of $\calb\calo$ to have a fixed K\"ahler orbifold structure. Thus, the morphisms of $\calb\calo$ will also be maps of the K\"ahler structures. We shall denote the objects in $\calb\calo_0$ by $\calk_{\bfm,\bfn}$.

\begin{theorem}\label{itjoinsasthm}
Let $M_{\bfl,\bfw}^{2n+1}$ be an iterated $S^3$-join of height $n$. Then for each $k\in\{2,\ldots,n\}$ the $S^1$ orbifold quotient of a quasi-regular Sasakian structure $\cals_{\bfl,\bfw,\bfv}$ with Reeb vector field $\xi_\bfv\in \gt^+_\bfw$ on $M_{\bfl,\bfw}$ is a K\"ahler Bott orbifold $\calk_{\bfm_k,\bfn_k}$ of the form $(M_k(A(k)),\grD^T_\bfm,\gro_k)$ with $\grD^T_\bfm$ given by Equation \eqref{Bottbranchdiv} and $A(k)$ has the block form
$$\Bigl(
\begin{array}{c|c}
A(k-1) & 0 \\
\hline
A(k)^1_k \cdots A(k)^{k-1}_k & 1
\end{array}
\Bigr)$$
where $A(k)$ is a matrix representative for the Bott manifold $M_{k}$. Moreover, the line orbibundle $L_k$ defining the Bott tower $(M_k(A(k)),\grD^T_{\bfm_k},\gro_k)\fract{\pi}{\longrightarrow}(M_{k-1}(A(k-1)),\grD^T_{\bfm_{k-1}},\gro_{k-1})$ satisfies 
\begin{equation}\label{c1Kah}
c_1(L_k^{\Upsilon_{k-1}})=\sum_{j=1}^{k-1}A(k)^j_kx_j=n_k[\gro_{k-1}]_I, 
\end{equation}
and is determined by the Reeb vector field $\xi_{\bfv_{k-1}}$ where $\bfm_k=m_k(v^0_k,v^\infty_k)$,
\begin{equation}\label{Kahclasseqn}
m_k =\frac{l^\infty_{k-1}}{s_k}, \qquad n_k=\frac{l^0_{k-1}}{s_k}(w_k^0v_k^\infty-w_k^\infty v_k^0), \qquad s_k=\gcd(l^\infty_{k-1},|w^0_kv^\infty_k-w^\infty_kv^0_k|).
\end{equation}
\end{theorem}

\begin{proof}
The proof is by induction. At each stage the quasi-regular Sasakian structure has a Reeb vector field given by $\xi_{\bfv_k}\in\gt^+_{\bfw_k}$. For $k=1$ the Bott orbifold is $\bbc\bbp^1[\bfw_1]=(\bbc\bbp^1,\grD_{\bfw_1})$ with $w_1^0,w_1^\infty$ coprime, i.e. $m_1=1$.  Assume that the theorem is true for $k-1$. Then, it follows from the definition of the join construction that for any (quasi-regular) Sasakian structure in $\gt^+_\bfw$ we can associate the Sasakian structure $\cals_\bfm$. Then it follows from Equation \eqref{Bottbranchdiv}, and the form of $A(k)$ that the hypothesis holds for $k$. Moreover, the relation between the orbifold K\"ahler classes $\pi^*[\gro_{k-1}]$ and $[\gro_k]$ is given by Lemma 2.11 of \cite{BoTo19a}. Thus, the integral K\"ahler class $[\gro_k]_I$ on $M_{k}(A_{k})$ satisfies the recursion relation
\begin{equation}\label{KahclassBott}
[\gro_k]_I= m_kl_{k-1}^0w_k^0v_k^\infty \pi^*[\gro_{k-1}]_I +m_ks_k\Upsilon_{\cals_{k-1}}x_k.
\end{equation}
Since the CR structure is fixed $\eta_k$ uniquely determines $\gro_k$.
The result now follows from Theorem 2.7 of \cite{BoTo19a}.
\end{proof}

We do not need to choose the Reeb vector field $\xi_{\bfv_n}$ at the last stage to be quasi-regular.

\begin{remark}\label{KSremark}
Not all Bott orbifolds occur as quotients of quasi-regular Reeb vector fields in the $\bfw$ Sasaki cone of an iterated join. Equation \eqref{c1Kah} places restrictions on the entrees of the matrix $A$. Proposition 4.13 and Example 4.14 of \cite{BHLT16} describe explicit examples of this general phenomenon. The well known Koiso-Sakane manifold \cite{KoSa86}, which is a special case of a 
Bott manifold (\cite{BoCaTo17}, page 53), gives a particular example.
\end{remark}

Theorem \ref{itjoinsasthm} gives rise to a functor $F^k:\cals\calj^k\rightrightarrows \calb\calo^k$ defined by $F^k(\cals_{\bfl_k,\bfw_k,\bfv_k})=\calk_{\bfm_k,\bfn_k}$ with the relations \eqref{Kahclasseqn} understood. We shall often drop the superscripts and subscripts $k$ when it is clear from the context that we are fixing the stage $k$. The Koiso-Sakane Bott manifold shows that this functor is not essentially surjective, but it is left invertible since $F(f)$ is in $\calb\calo_1$ and the composition rules of a (covariant) functor hold. One can check that $\cals\calj$ is a {\it fibered category} \cite{Vis05} over $\calb\calo$ in the sense that given morphisms $f,\tf'\in\cals\calj_1$ there exists a unique morphism $\tf$ such that the diagram

\begin{equation}\label{cartfib}
\xymatrix{
\cals_{\tilde{\bfl},\tilde{\bfw},\tilde{\bfv}} \ar[d]_F  \ar@/^/[drr]^{\tf'}
\ar@{.>}[dr]_\tf |-{} \\
\calk_{\tilde{\bfm},\tilde{\bfn}} \ar[dr]_h \ar@/^/[drr] & \cals_{\bfl,\bfw,\bfv} \ar[d]_F \ar[r]_f
& \cals_{\bfl',\bfw',\bfv'} \ar[d]_F \\
& \calk_{\bfm,\bfn} \ar[r]^{F(f)} & \calk_{\bfm',\bfn'} }
\end{equation}
commutes. The morphism $f$ is called a {\it Cartesian} morphism. Given $\calk_{\bfm,\bfn}\in \calb\calo_0$ we denote by $\cals\calj(\calk_{\bfm,\bfn})$ the fiber over the object $\calk_{\bfm,\bfn}$ which is a subcategory of $\cals\calj$ whose objects are the objects $\cals_{\bfl,\bfw,\bfv}$ of $\cals\calj$ such that $F(\cals_{\bfl,\bfw,\bfv})= \calk_{\bfm,\bfn}$ and whose morphisms $f$ satisfy $F(f)={\rm id}_{\calk_{\bfm,\bfn}}$. By Proposition 2.21 of \cite{BoTo19a} that $F$ has a left inverse at each stage $k$. This implies that the subcategory $F(\cals_{\bfl,\bfw,\bfv})$ has only one object, and one morphism, the identity. However, there is another subcategory whose set $F(\cals_{\bfl,\bfw,\bfv})_0$ of objects is a single point $\{\cals_{\bfl,\bfw,\bfv}\}$, but whose morphisms are the morphisms of $\cals\calj^k$ whose source and target are $\{\cals_{\bfl,\bfw,\bfv}\}$. This is precisely the isotropy subgroup $\gA\gu\gt(\cals_{\bfl,\bfw,\bfv})$ at $\cals_{\bfl,\bfw,\bfv}\in\cals\calj_0$, and $F$ maps these morphisms to the isotropy subgroup $\gA\gu\gt(\calk_{\bfm,\bfn})$ at $\calk_{\bfm,\bfn}\in\calb\calo_0$. Note that Remark \ref{KSremark} shows that $\cals\calj(\calk_{\bfm,\bfn})$ can be empty.

\begin{example}\label{fibinv}
For any element $\xi_\bfv\in\gt^+_\bfw$ we consider the involution sending $\bfv=(v^0,v^\infty)$ to $\bfv^\perp=(v^\infty,v^0)$.
This induces an involution of joins $M_{\bfl_k,\bfw_k}\leftrightarrows M_{\bfl_k,\bfw^\perp_k}$ at stage $k$ giving rise to bijections $\cals_{\bfl_k,\bfw_k,\bfv_k}\leftrightarrows \cals_{\bfl_k,\bfw^\perp_k,\bfv^\perp_k}$. Applying the functor $F$ gives the $k$th fiber inversion map $\grt_k\in \calb\calo^k_1$ for the corresponding Bott orbifolds.
\end{example}


\subsection{Constant Scalar Curvature}
The natural setting for the join construction is the orbifold category \cite{BGO06,BHLT16}, and the basic theorem on constant scalar curvature Sasaki (cscS) metrics is Theorem 3.2 of \cite{BoTo19a}:

\begin{theorem}\label{admjoincsc}
Let $M_{\bfl,\bfw}=M\star_\bfl S^3_\bfw$ be the $S^3_\bfw$-join with a quasi-regular Sasaki orbifold $M$ which is an $S^1$ orbibundle over a compact K\"ahler orbifold $N$ with constant scalar curvature $s_N$. Then for each vector $\bfw=(w^0,w^\infty)\in \bbz^+\times\bbz^+$ with relatively prime components satisfying $w^0>w^\infty$ there exists a Reeb vector field $\xi_\bfv$ in the 2-dimensional $\bfw$-Sasaki cone on $M_{\bfl,\bfw}$ such that the corresponding ray of Sasakian structures $\cals_a=(a^{-1}\xi_\bfv,a\eta_\bfv,\Phi,g_a)$ has constant scalar curvature. Moreover, if $s_N\geq 0$, then the $\bfw$-Sasaki cone $\gt^+_\bfw$ is exhausted by extremal Sasaki metrics. If $s_N>0$ and $l^\infty$ is sufficiently large then the $\bfw$-cone has at least 3 cscS rays.
\end{theorem}

We can easily apply this theorem inductively to obtain cscS orbifold metrics on the iterated joins \eqref{itjoin}; however, in order to obtain smooth iterated joins $M_{\bfl,\bfw}$ the gcd conditions of Lemma \ref{smoothlem} must hold. These conditions become increasingly complex since the order of the orbifolds become increasingly complex.

Since the $\bfw$-cone $\gt^+_\bfw$ has dimension 2, the cardenality of the set of cscS metrics in $\gt^+_\bfw$ is finite (\cite{BHL17}, Corollary 1.7). Moreover, these can be obtained from the admissible construction as real roots $b$ of the polynomial \cite{BoTo14a,BoTo19a}
\begin{equation}\label{functionf}
\begin{array}{ccl}
f(b) & = &  (w^0)^{2(d_N+1)} b^{2 d_N+3}( A_N l^\infty  + l^0  (d_N + 1)  w^\infty-b (d_N + 1 ) l^0  w^0 )\\
\\
& - & (w^0)^{d_N+2}  (w^\infty)^{d_N} b^{d_N + 3}  ((d_N+1)  (A_N ( d_N+1) l^\infty  - l^0  (( d_N+1) w^0 + ( d_N+2) w^\infty)))\\
\\
& + & (w^0)^{d_N+1}  (w^\infty)^{d_N+1}  b^{d_N + 2}  (2 A_N d_N (d_N+2) l^\infty  - (d_N+1)(2d_N+3) l^0  (w^0 + w^\infty))\\
\\
& - &  (w^0)^{d_N}  (w^\infty)^{d_N+2}  b^{d_N + 1}(d_N+1) (A_N ( d_N+1) l^\infty  - l^0  ((d_N+2) w^0 + ( d_N+1) w^\infty))\\
\\
& + & (w^\infty)^{2 (d_N + 1)}( b (A_N l^\infty  + l^0  (d_N + 1) w^0)-( d_N + 1 ) l^0  w^\infty)
\end{array}
\end{equation}
where $d_N$ is the complex dimension of $N$. In our case $N$ is a Bott orbifold with a cscK metric chosen with a rational K\"ahler class,  so $A_N$ is a rational number in this case. Thus, for any pair $(w^0,w^\infty)$ we can take $b$ to be a large enough rational number and solve for $l^\infty/l^0$. This will give cscS metrics for any dimension. Note that it follows from page 1052 of \cite{BoTo14a} that $b=\frac{w^\infty}{w^0}$ is a root of multiplicity 3. This root corresponds to the quotient having a product structure which does not have constant scalar curvature unless $\bfw=(1,1)$ which we exclude. Thus, we are left with a polynomial $g(b)$ of degree $2d_N+1$ which implies\footnote{This lemma holds for any $S^3_\bfw$ join.} 
\begin{lemma}\label{cardcscwcone}
There are at most $2d_N+1$ cscS rays in $\gt^+_\bfw$. 
\end{lemma}

However,  we cannot a priori guarantee smoothness as we increase dimension. Obtaining quasi-regular Sasaki metrics of constant scalar curvature on smooth iterated joins becomes quite difficult. We shall investigate this further in the Gorenstein case in Section \ref{Gorensteinsect} below. For now we consider the beginning stage, namely $S^3$ bundles over $S^2$. It is well known \cite{BoPa10,BHLT16} that every toric contact structure on an $S^3$ bundle over $S^2$ comes from an $S^3$ join of the form $S^3_{\bfw_1}\star_\bfl S^3_{\bfw_2}$. Here for Theorem \ref{admjoincsc} to be applicable we consider those toric contact structures whose Sasaki automorphism group $\gA\gu\gt(\cals)$ contains $SU(2)\times \bbt^2$. These are precisely the $S^3_\bfw$ join with the standard $S^3$.

\begin{example}[$S^3$ bundles over $S^2$]
At stage 2, i.e. $k=2$ the join $M_{\bfl,\bfw}$ is smooth if and only if $\gcd(l^\infty,w^0w^\infty)=1$ (we always assume that $l^0,l^\infty$ and $w^0,w^\infty$ are pairwise relatively prime).
Consider the diagram
\begin{equation}\label{s2comdia2}
\begin{matrix}  &&S^3\times S^3_\bfw &&& \\
                          && \decdnar{\pi_{L}}  &&& \\
                           &&M_{\bfl,\bfw} &&&\\
                           & \swarrow\pi_R && \searrow\pi_\bfv & \\
                                                      &&&&& \\
                           \bbc\bbp^1\times\bbc\bbp^1[\bfw] &&&&\calh_a(\grD_\bfm) 
\end{matrix} 
\end{equation}
where $\calh_a(\grD_\bfm)$ is a Hirzebruch orbifold. The notation here is that $S^3$ is the standard sphere with its diagonal $S^1$ action and standard $\bbc\bbp^1$ quotient, $S^3_\bfw$ is the sphere with a weighted $S^1$ action with weight vector $\bfw=(w^0,w^\infty)$ and quotient the weighted projective line $\bbc\bbp^1[\bfw]$. The map $\pi_R\circ\pi_L$ is just the product of $S^1$ bundles; whereas, the map $\pi_\bfv \circ\pi_L$ depends on a choice of quasiregular Reeb vector field $\xi_\bfv$ in the $\bfw$ cone $\gt^+_\bfw$. Here $\grD_\bfm$ is a branch divisor with ramification indices $\bfm=(m^0,m^\infty)$. The idea is to end up with a Reeb vector field $\xi_\bfv$ of constant scalar curvature on $M_{\bfl,\bfw}$. Of course, this is a very special case of Theorem 2.7 of \cite{BoTo19a} , Theorem 3.8 of \cite{BoTo14a}, and Theorem \ref{itjoinsasthm}. From these theorems we see that 
\begin{equation}\label{joinparam1}
\bfm=\frac{l^\infty}{s}\bfv,\quad s=\gcd(l^\infty,|w^0v^\infty-w^\infty v^0|), \quad a=l^0\bigl(\frac{w^0v^\infty-w^\infty v^0}{s}\bigr).
\end{equation}
We also have Equation (9) of \cite{BoTo14a}
\begin{equation}\label{c1join}
c_1(\cald_{\bfl,\bfw})=(2l^\infty-l^0|\bfw|)\grg
\end{equation}
where $\grg$ is a positive generator of $H^2(M_{\bfl,\bfw},\bbz)$ with respect to the oriented vector bundle $\cald_{\bfl,\bfw}$.
So there are two diffeomorphism types depending on whether $l^0|\bfw|$ is even or odd which correspond to the trivial bundle $S^2\times S^3$ and the non-trivial $S^3$ bundle over $S^2$, respectively.  Within each diffomorphism type there is a countably infinite number of inequivalent toric contact structures of Sasaki type. 

The cscS metrics are determined by the roots of the polynomial (34) in \cite{BoTo19a}, and a special case of \eqref{functionf}, namely the polynomial
\begin{equation}\label{functionfd1}
f(b)  =  2(-bw^0+w^\infty)^3\bigl(b^3l^0(w^0)^2+b^2(2l^0w^0w^\infty-l^\infty w^0)-b(2l^0w^0w^\infty-l^\infty w^\infty)-l^0(w^\infty)^2\bigr). 
\end{equation}
Note that the admissible construction requires that $b\neq w^\infty/w^0$, so we are dealing with the roots of a cubic polynomial which implies that there are at most 3 cscS rays in the $\bfw$ cone $\gt^+_\bfw$. In fact we have

\begin{proposition}\label{numcscS}
Let $M_{\bfl,\bfw}$ be the join $S^3\star_\bfl S^3_\bfw$ and assume that $w^0>w^\infty$. Then the $\bfw$ cone $\gt^+_\bfw$ of $M_{\bfl,\bfw}$ has at least one cscS metric and at most three cscS metrics. Furthermore, there exists $L_{(l^0,w^0,w^\infty)}\in \bbr^+$ such that for $l^\infty < L_{(l^0,w^0,w^\infty)}$ there is precisely one cscS metric in $\gt^+_\bfw$, whereas for $l^\infty > L_{(l^0,w^0,w^\infty)}$ there are precisely three cscS metrics $\gt^+_\bfw$, and $L_{(l^0,w^0,w^\infty)}$ satisfies the inequality 
$$2l^0w^0<L_{(l^0,w^0,w^\infty)}< \frac{11}{2}l^0w^0+5l^0 \frac{(w^0-w^\infty)}{2}.$$
\end{proposition}

\begin{proof}
It is convenient to write $f(b)=2(bw^0-w^\infty)^3g(b)$ with
$$g(b):= -l^0(w^0)^2b^3+(l^\infty -2l^0w^\infty) w^0b^2-(l^\infty -2l^0w^0)w^\infty b+l^0(w^\infty)^2.$$
Then each real positive root of $g$ corresponds to a cscS ray in the $\bfw$ cone $\gt^+_\bfw$. Since $g(0)>0$ and $\displaystyle\lim_{b\rightarrow +\infty}g(b)=-\infty$ we confirm that $g(b)$ has at least one positive real root, a fact that follows from Theorem 3.3 of \cite{BoTo19a}. Furthermore, for $l^\infty \leq 2l^0w^0$, Descartes' rule of signs tells us that $g(b)$ has at most one real positive root. Now assume $l^\infty > 2l^0w^0$ in which case Descartes' rule of sign tells us that $g(b)$ has no negative real roots, so any real root of $g(b)$ must be positive. Thus, $g(b)$ will have three distinct real positive roots if and only if the discriminant of $g(b)$ is positive. This discriminant is calculated to be
$(w^0w^\infty)^2 h$, where $h$ may be viewed as a polynomial of degree $4$ in $l^\infty$:
$$
\begin{array}{ccl}
h &=& (l^\infty)^4\\
\\
&-&8 l^0(w^0+w^\infty)(l^\infty)^3\\
\\
&+&2(l^0)^2(2(w^0)^2+41w^0w^\infty+2(w^\infty)^2)(l^\infty)^2\\
\\
&-&100(l^0)^3w^0w^\infty(w^0+w^\infty)l^\infty\\
\\
&+& (l^0)^4w^0w^\infty(32(w^0)^2+61w^0w^\infty+32(w^\infty)^2).
\end{array}
$$
Now, we may calculate that the discriminant of $h$ as a polynomial in $l^\infty$ equals \newline
$-768(l^0)^{12}w^0(w^0-w^\infty)^4w^\infty(8w^0+w^\infty)^3(w^0+8w^\infty)^3$, which is clearly negative. Thus $h$ has precisely two real roots. Since
$h\mid_{l^\infty=0} >0$, $h\mid_{l^\infty=2l^0w^0} = -(l^0)^4w^0(32(w^0-w^\infty)^3+27(w^\infty)^2(w^0-w^\infty)+27(w^\infty)^3)<0$,
and $\displaystyle\lim_{l^\infty \rightarrow +\infty}h>0$, we know that $h$ has precisely one real root in the interval
$(0,2l^0w^0)$ and precisely one real root in $(2l^0w^0, +\infty)$. Putting this together we can conclude that for fixed values $(l^0,w^0,w^\infty)$, there exists a real positive value $L_{(l^0,w^0,w^\infty)}$ such that for $l^\infty < L_{(l^0,w^0,w^\infty)}$, $g(b)$ has precisely one real positive root, and for $l^\infty > L_{(l^0,w^0,w^\infty)}$,  $g(b)$ has three distinct real positive roots. 
The above observations tell us that $L_{(l^0,w^0,w^\infty)}>2l^0w^0$. Further, since
\begin{itemize}
\item $g(0)>0$,
\bigskip
\item $g(\frac{w^\infty}{2w^0})=\frac{-(w^\infty)^2(2l^\infty-l^0(16w^0-5w^\infty)}{8w^0}$,
\bigskip
\item $g(\frac{w^\infty}{w^0})=  \frac{3l^0(w^\infty)^2(w^0-w^\infty)}{w^0}>0$, and
\bigskip
\item $\lim_{b\rightarrow +\infty}g(b)=-\infty,$ 
\end{itemize}
we know that $L_{(l^0,w^0,w^\infty)}< \frac{l^0(16w^0-5w^\infty)}{2}$.
\end{proof}

\begin{remark}
Note that $L_{(l^0,w^0,w^\infty)}$ is the real root of $h$ (as a polynomial of $l^\infty$) appearing in the interval 
$(2l^0w^0, +\infty)$. It is likely not an integer, but in the event that it is, we can say that for $l^\infty = L_{(l^0,w^0,w^\infty)}$, $g(b)$ has either one (third order) real positive root or two positive real roots (one of order one and one of order two).
\end{remark}

Proposition \ref{numcscS} implies that the first Chern class of the contact bundle is an obstruction to the existence of multiple cscS rays in $\gt^+_\bfw$. In fact,

\begin{corollary}\label{c1cscinv}
Assume the hypothesis of Proposition \ref{numcscS} and suppose $\gt^+_\bfw$ has more than one cscS ray. Then 
$$c_1(\cald_{\bfl,\bfw}) >\bigl(2l^0w^0+l^0(w^0-w^\infty)\bigr)\grg>0.$$
\end{corollary}

\end{example}

\subsection{Smoothness in the Gorenstein Case, $c_1(\cald)=0$}\label{Gorensteinsect}
In this case Theorem \ref{admjoincsc} produces a Sasaki-$\eta$-Einstein  ray with a unique Sasaki-Einstein metric (see Theorem 3.3 of \cite{BoTo19a} for the precise statement). So in this case we assume that the Sasakian structure for each intermediate stage is a quasi-regular Sasaki-Einstein structure.

\begin{example}[Dimension Five]
In this case the beginning stage 2 gives the $Y^{p,q}$ which are well understood \cite{GMSW04a,BG05,BoTo14a}. They give a countable infinity \cite{AbMa10,BoPa10,AbMa18} of toric contact structures on $S^2\times S^3$ each of whose $\bfw$ cone contains a unique Sasaki-Einstein metric. In this case smoothness is automatic since $l^\infty=\frac{w^0+w^\infty}{\gcd(2,w^0+w^\infty)}$.
\end{example}

\begin{example} [Dimension Seven]
The stage 3 case was considered in \cite{BoTo18c}. Here one chooses a quasi-regular Reeb vector field $\xi_{\bfm_2}$ on a $Y^{p,q}$ with a cscS metric. This puts constraints on $p,q$, namely that $4p^2-3q^2=n^2$ for some $n\in\bbz$. Then Lemma 2.1 of \cite{BoTo18c} says that the join $M_{\bfl,\bfw}^7=Y^{p,q}\star_\bfl S^3_\bfw$ is smooth if and only if $\gcd(l^\infty m_2v_2^0v^\infty_2,l^0w^0w^\infty)=1$ where $\bfm_2=m_2(v^0,v^\infty)$ with 
$m_2=\frac{p}{\gcd(p,|\frac{p+q}{l^0}v^\infty_2-\frac{p-q}{l^0}v^0_2|)}$. We gave examples of quasi-regular SE metrics on smooth 7-manifolds which include the infinite family
\begin{equation}\label{infse7man}
\cals^t_3=Y^{780300 t^2+65790 t+1387,15 (170 t+7) (306 t+13)}\star_{306\,t+ 13,4} S^3_{17,3}
\end{equation}
for $t\in\bbz^+$.
\end{example}

\begin{example}[Higher Dimension]
Instinctively it should be true that one can use the iterative join in a non-trivial way to construct quasi-regular smooth Sasaki Einstein structures of arbitrarily high dimension. The following result is meant to offer a hint on both why this is a reasonable belief and also why this would be very challenging to verify in general.

\begin{proposition}\label{911prop}
In dimensions $9$ and $11$ there exist countably infinite families of quasi-regular smooth Sasaki Einstein structures in the form of non-trivial
iterated $S_\bfw^3$-joins $M^{2k+1}_{\bfl,\bfw}$
\end{proposition}

\begin{proof}
The starting point is the one-parameter family of quasi-regular smooth, $7$-dimensional Sasaki-Einstein structures $\{\cals^t_3\}$, $t \in \bbz^+$, given by Equation \eqref{infse7man}. As discussed in \cite{BoTo18c}, these are obtained by iterating the join construction twice and the quasi-regular quotient is a stage $3$ Bott orbifold.
Now we observe - using Lemma 2.9 and (33), respectively from \cite{BoTo19a} - that for a given parameter $t\in\bbz^+$, we have
the orbifold order and Fano index, respectively
$$
\begin{array}{ccl}
\Upsilon_3 &=& 2^2\cdot3^2\cdot 17 \cdot (780300\, t^2 + 65790\, t+1387)(1020\,t+43)(255\,t+ 11),\\
\\
\cali_{\bfv_3}&=&13.
\end{array}
$$

Assuming that we choose the quasi-regular Sasaki Einstein structure $\cals^t_3$ such that the transverse K\"ahler class is primitive orbifold our
next step is to look for a smooth join $\cals^t_3 \star_{\bfl_4} S^3_{\bfw_4}$ with a quasi-regular $\eta$-Einstein ray in the $\bfw$-cone.
This is done experimentally using the machinery presented in e.g. Section 4 of \cite{BoTo19a}. Indeed, let $k=2$ in
$p^\pm(k)$ (with $d=3$) as in Lemma 4.1 of \cite{BoTo19a} with
$v_4^\infty/v_4^0=p^-(k)/p^+(k)$ and $w_4^\infty/w_4^0=p^-(k)/(kp^+(k))$. Then we get
$$
\begin{array}{ccl}
(v_4^0,v_4^\infty) & = & (49,26)\\
\\
(w_4^0,w_4^\infty) & = & (49,13).\\
\end{array}
$$
Setting
$$
\begin{array}{ccccl}
l_3^0 & = &\frac{\cali_{\bfv_{3}}}{\gcd(w_4^0+w_4^\infty,\cali_{\bfv_{3}})}&=&  13\\
\\
l_3^\infty & = & \frac{w_4^0+w_4^\infty}{\gcd(w_4^0+w_4^\infty,\cali_{\bfv_{3}})}&=&  62\\
\end{array}
$$
we know that $(v_4^0,v_4^\infty)$ represents a quasi-regular $\eta$-Einstein ray in the $\bfw$-cone of $\cals^t_3 \star_{\bfl_4} S^3_{\bfw_4}$.
Further, $\cals^t_3 \star_{\bfl_4} S^3_{\bfw_4}$ will be smooth exactly when
$\gcd(l_3^\infty \Upsilon_3, l_3^0 w_4^0 w_4^\infty)=1$, i.e., exactly when
$$\gcd(2 \cdot 3 \cdot 17 \cdot 31 \cdot (780300\, t^2 + 65790\, t+1387)(1020\,t+43)(255\,t+ 11), 7 \cdot 13 )=1,$$
which in turn is equivalent to
$$\gcd((780300\, t^2 + 65790\, t+1387)(1020\,t+43)(255\,t+ 11), 7 \cdot 13 )=1.$$
Noting that $1387=19\cdot 73$, $43$, and $11$ are all coprime with $7 \cdot 13$, we realize that this is satisfied if
we assume $t=7\cdot 13\cdot \hat t$ with $\hat t \in \bbz^+$. With these choices we now have a one-parameter family of quasi-regular smooth, $9$-dimensional Sasaki-Einstein structures $\{\cals^{\hat t}_4\}$, $\hat t \in \bbz^+$ and we can calculate that

$$
\begin{array}{ccl}
\Upsilon_4 &=& 2^4\cdot3^2\cdot 7^2 \cdot 13 \cdot 17 \cdot 31 \cdot (6461664300\, {\hat t}^2+ 5986890\, {\hat t} +1387)(92820\,{\hat t}+43)
(23205\,{\hat t}+ 11)\\
\\
\cali_{\bfv_4}&=&150.
\end{array}
$$
Assuming that we choose the quasi-regular Sasaki Einstein structure $\cals^{\hat t}_4$ such that the transverse K\"ahler class is a primitive orbifold class, our
next step is to look for a smooth join $\cals^{\hat t}_4 \star_{\bfl_4} S^3_{\bfw_5}$ with a quasi-regular $\eta$-Einstein ray in the $\bfw$-cone.
This is done experimentally using the same machinery as above. After much trial and error and some instinctive avoidance of those prime numbers that have already arrived at the scene, we proceed with $k=37/5$ in
$p^\pm(k)$ (with $d=4$) as in Lemma 4.1 of \cite{BoTo19a} with
$v_5^\infty/v_5^0=p^-(k)/p^+(k)$ and $w_5^\infty/w_5^0=p^-(k)/(kp^+(k))$, to get
$$
\begin{array}{ccl}
(v_5^0,v_5^\infty) & = & (3498805,834997)\\
\\
(w_5^0,w_5^\infty) & = & (25891157,834997).\\
\end{array}
$$
Now we set
$$
\begin{array}{ccccl}
l_4^0 & = &\frac{\cali_{\bfv_{4}}}{\gcd(w_5^0+w_5^\infty,\cali_{\bfv_{4}})}&=& 25  \\
\\
l_4^\infty & = & \frac{w_5^0+w_5^\infty}{\gcd(w_5^0+w_5^\infty,\cali_{\bfv_{4}})}&=& 7 \cdot 13 \cdot 31 \cdot 1579 \\
\end{array}
$$
(where the factors displayed for $l_4^\infty$ are all prime).
Then $(v_5^0,v_5^\infty)$ represents a quasi-regular $\eta$-Einstein ray and the join $\cals^{\hat t}_4 \star_{\bfl_4} S^3_{\bfw_5}$ will be smooth exactly when
$$\gcd(l_4^\infty\Upsilon_4, l_4^0 w_5^0 w_5^\infty)=1.$$
Now, one can calculate that the left side of this equation equals
$$\gcd((6461664300\, {\hat t}^2+ 5986890\, {\hat t} +1387)(92820\,{\hat t}+43)
(23205\,{\hat t}+ 11), 5^2\cdot 29\cdot 37 \cdot 28793 \cdot 699 761)$$
and since $\gcd(1387\cdot 43 \cdot 11, 5^2\cdot 29\cdot 37 \cdot 28793 \cdot 699 761)=1$, we can
make sure $\gcd(l_4^\infty\Upsilon_4, l_4^0 w_5^0 w_5^\infty)=1$ by simply assuming that
$$\hat t = 5^2\cdot 29\cdot 37 \cdot 28793 \cdot 699 761\, \tilde{t},$$
for some $\tilde t \in \bbz^+$. We thus arrive at a one-parameter family of quasi-regular smooth, $11$-dimensional Sasaki-Einstein structures $\{\cals^{\tilde t}_5\}$, $\tilde t \in \bbz^+$.
\end{proof}

It seems very likely that one can repeat this process in perpetuity (without resorting to products), but it is not clear how to show this explicitly and it was certainly more complicated to arrive at $\{\cals^{\tilde t}_5\}$ than at any of the previous stages. Nevertheless, we propose

\begin{conjecture}\label{itjoinseconj}
In the Gorenstein case the iterated $S^3$ joins admit smooth Sasaki-Einstein manifolds in all odd dimensions.
\end{conjecture}

It seems reasonable that an analogous result should hold in the general cscS case, although this would be even more complex.

\end{example}

\def\cprime{$'$} \def\cprime{$'$} \def\cprime{$'$} \def\cprime{$'$}
  \def\cprime{$'$} \def\cprime{$'$} \def\cprime{$'$} \def\cprime{$'$}
  \def\cdprime{$''$} \def\cprime{$'$} \def\cprime{$'$} \def\cprime{$'$}
  \def\cprime{$'$}
\providecommand{\bysame}{\leavevmode\hbox to3em{\hrulefill}\thinspace}
\providecommand{\MR}{\relax\ifhmode\unskip\space\fi MR }
\providecommand{\MRhref}[2]{%
  \href{http://www.ams.org/mathscinet-getitem?mr=#1}{#2}
}
\providecommand{\href}[2]{#2}


\begin{thebibliography}{BHLTF18}

\bibitem[AM12]{AbMa10}
Miguel Abreu and Leonardo Macarini, \emph{Contact homology of good toric
  contact manifolds}, Compos. Math. \textbf{148} (2012), no.~1, 304--334.
  \MR{2881318}

\bibitem[AM18]{AbMa18}
\bysame, \emph{On the mean {E}uler characteristic of {G}orenstein toric contact
  manifolds}, Int. Math. Res. Not. IMRN, doi:10.1093 (2018).

\bibitem[BCTF19]{BoCaTo17}
Charles~P. Boyer, David M.~J. Calderbank, and Christina~W.
  T{\o}nnesen-Friedman, \emph{The {K}\"{a}hler geometry of {B}ott manifolds},
  Adv. Math. \textbf{350} (2019), 1--62. \MR{3945589}

\bibitem[BG08]{BG05}
Charles~P. Boyer and Krzysztof Galicki, \emph{Sasakian geometry}, Oxford
  Mathematical Monographs, Oxford University Press, Oxford, 2008. \MR{MR2382957
  (2009c:53058)}

\bibitem[BGO07]{BGO06}
Charles~P. Boyer, Krzysztof Galicki, and Liviu Ornea, \emph{Constructions in
  {S}asakian geometry}, Math. Z. \textbf{257} (2007), no.~4, 907--924.
  \MR{MR2342558 (2008m:53103)}

\bibitem[BHL18]{BHL17}
Charles Boyer, Hongnian Huang, and Eveline Legendre, \emph{An application of
  the {D}uistermaat--{H}eckman theorem and its extensions in {S}asaki
  geometry}, Geom. Topol. \textbf{22} (2018), no.~7, 4205--4234. \MR{3890775}

\bibitem[BHLTF18]{BHLT16}
Charles~P. Boyer, Hongnian Huang, Eveline Legendre, and Christina~W.
  T{\o}nnesen-Friedman, \emph{Reducibility in {S}asakian geometry}, Trans.
  Amer. Math. Soc. \textbf{370} (2018), no.~10, 6825--6869. \MR{3841834}

\bibitem[BP14]{BoPa10}
Charles~P. Boyer and Justin Pati, \emph{On the equivalence problem for toric
  contact structures on {$S^3$}-bundles over {$S^2$}}, Pacific J. Math.
  \textbf{267} (2014), no.~2, 277--324. \MR{3207586}

\bibitem[BTF16]{BoTo14a}
Charles~P. Boyer and Christina~W. T{\o}nnesen-Friedman, \emph{The {S}asaki
  join, {H}amiltonian 2-forms, and constant scalar curvature}, J. Geom. Anal.
  \textbf{26} (2016), no.~2, 1023--1060. \MR{3472828}

\bibitem[BTF19a]{BoTo19a}
\bysame, \emph{The ${S}^3$ {S}asaki join construction}, arXiv:1911.11031
  (2019).

\bibitem[BTF19b]{BoTo18c}
\bysame, \emph{Sasaki-{E}instein metrics on a class of 7-manifolds}, J. Geom.
  Phys. \textbf{140} (2019), 111--124. \MR{3923473}

\bibitem[BTF20]{BoTo15}
\bysame, \emph{On positivity in {S}asaki geometry}, Geom. Dedicata \textbf{204}
  (2020), 149--164. \MR{4056696}

\bibitem[Don12]{Don12}
S.~K. Donaldson, \emph{K\"ahler metrics with cone singularities along a
  divisor}, Essays in mathematics and its applications, Springer, Heidelberg,
  2012, pp.~49--79. \MR{2975584}

\bibitem[GK94]{GrKa94}
Michael Grossberg and Yael Karshon, \emph{Bott towers, complete integrability,
  and the extended character of representations}, Duke Math. J. \textbf{76}
  (1994), no.~1, 23--58. \MR{1301185 (96i:22030)}

\bibitem[GMSW04]{GMSW04a}
J.~P. Gauntlett, D.~Martelli, J.~Sparks, and D.~Waldram,
  \emph{Sasaki-{E}instein metrics on {$S^2\times S^3$}}, Adv. Theor. Math.
  Phys. \textbf{8} (2004), no.~4, 711--734. \MR{2141499}

\bibitem[Gol84]{Gol84}
Robert Goldblatt, \emph{Topoi}, second ed., Studies in Logic and the
  Foundations of Mathematics, vol.~98, North-Holland Publishing Co., Amsterdam,
  1984, The categorial analysis of logic. \MR{766560}

\bibitem[Joh77]{Joh77}
P.~T. Johnstone, \emph{Topos theory}, Academic Press [Harcourt Brace
  Jovanovich, Publishers], London-New York, 1977, London Mathematical Society
  Monographs, Vol. 10. \MR{0470019}

\bibitem[KS86]{KoSa86}
Norihito Koiso and Yusuke Sakane, \emph{Nonhomogeneous {K}\"ahler-{E}instein
  metrics on compact complex manifolds}, Curvature and topology of {R}iemannian
  manifolds ({K}atata, 1985), Lecture Notes in Math., vol. 1201, Springer,
  Berlin, 1986, pp.~165--179. \MR{859583 (88c:53047)}

\bibitem[Laz04]{Laz04a}
R.~Lazarsfeld, \emph{Positivity in algebraic geometry. {I}}, Ergebnisse der
  Mathematik und ihrer Grenzgebiete. 3. Folge. A Series of Modern Surveys in
  Mathematics [Results in Mathematics and Related Areas. 3rd Series. A Series
  of Modern Surveys in Mathematics], vol.~48, Springer-Verlag, Berlin, 2004,
  Classical setting: line bundles and linear series. \MR{2095471
  (2005k:14001a)}

\bibitem[Ler10]{Ler10}
Eugene Lerman, \emph{Orbifolds as stacks?}, Enseign. Math. (2) \textbf{56}
  (2010), no.~3-4, 315--363. \MR{2778793}

\bibitem[MLM94]{MaMo94}
Saunders Mac~Lane and Ieke Moerdijk, \emph{Sheaves in geometry and logic},
  Universitext, Springer-Verlag, New York, 1994, A first introduction to topos
  theory, Corrected reprint of the 1992 edition. \MR{1300636}

\bibitem[MM03]{MoMr03}
I.~Moerdijk and J.~Mr{\v{c}}un, \emph{Introduction to foliations and {L}ie
  groupoids}, Cambridge Studies in Advanced Mathematics, vol.~91, Cambridge
  University Press, Cambridge, 2003. \MR{2012261}

\bibitem[Moe02]{Moe02}
I.~Moerdijk, \emph{Orbifolds as groupoids: an introduction}, Orbifolds in
  mathematics and physics (Madison, WI, 2001), Contemp. Math., vol. 310, Amer.
  Math. Soc., Providence, RI, 2002, pp.~205--222. \MR{2004c:22003}

\bibitem[MP97]{MoPr97}
I.~Moerdijk and D.~A. Pronk, \emph{Orbifolds, sheaves and groupoids},
  $K$-Theory \textbf{12} (1997), no.~1, 3--21. \MR{98i:22004}

\bibitem[Suy18]{Suy18}
Yusuke Suyama, \emph{Fano generalized {B}ott manifolds}, arXiv:1811.06209
  (2018).

\bibitem[Vis05]{Vis05}
Angelo Vistoli, \emph{Grothendieck topologies, fibered categories and descent
  theory}, Fundamental algebraic geometry, Math. Surveys Monogr., vol. 123,
  Amer. Math. Soc., Providence, RI, 2005, pp.~1--104. \MR{2223406}

\end{thebibliography}
\end{document}